\documentclass[12pt]{amsart}

  \usepackage{graphicx}
  \usepackage{latexsym,amscd,amssymb,epsfig,a4}

  \DeclareMathOperator{\des}{{\rm des}}
  \DeclareMathOperator{\Hilb}{{\rm Hilb}}

  \def\RR{{\mathbb R}}
  \def\fC{{\mathfrak C}}
  \def\fL{{\mathfrak L}}
  \def\Rat{{\mathrm R}}
  \def\diag{{\mathrm diag}}

  \def\CC{{\mathbb C}}
  \def\NN{{\mathbb N}}
  \def\PP{{\mathbb P}}

  \def\Bas{{\mathcal B}}
  
  \def\aa{{\mathfrak a}}
  \def\ee{{\mathfrak e}}
  \def\uu{{\mathfrak u}}
  \def\vv{{\mathfrak v}}
  \def\supp{{\mathrm supp}}

  \theoremstyle{plain}
    \newtheorem{theorem}{Theorem}[section]
    \newtheorem{proposition}[theorem]{Proposition}
    \newtheorem{lemma}[theorem]{Lemma}
    \newtheorem{corollary}[theorem]{Corollary}
    
  \theoremstyle{definition}

    \newtheorem{remark}[theorem]{Remark}

  \numberwithin{equation}{section}

\begin{document}

  \title[Formal Power Series and Algebras]{The Veronese Construction for Formal Power Series and Graded Algebras}

  \author{Francesco Brenti}
  \address{Dipartimento di Matematica \\
           Universita' di Roma "Tor Vergata" \\
           Via della Ricerca Scientifica, 1 \\
           00133, Roma, Italy}

  \email{brenti@mat.uniroma2.it}

  \author{Volkmar Welker}
  \address{Fachbereich Mathematik und Informatik\\
           Philipps-Universit\"at Marburg\\
           35032 Marburg, Germany}
  \email{welker@mathematik.uni-marburg.de}

  \thanks{Both authors were partially supported by Ateneo Italo-Tedesco and DAAD through a Vigoni project.}



  \begin{abstract}
    Let $(a_n)_{n \geq 0}$ be a sequence of complex numbers such that
    its generating series satisfies $\sum_{n \geq 0} a_nt^n = \frac{h(t)}{(1-t)^d}$ for some polynomial $h(t)$.
    For any $r \geq 1$ we study the transformation of the coefficient series
    of $h(t)$ to that of $h^{\langle r \rangle}(t)$ where $\sum_{n \geq 0} a_{nr} t^n =
    \frac{h^{\langle r \rangle}(t)}{(1-t)^d}$. We give a precise description of
    this transformation and show that under some natural mild hypotheses the roots
    of $h^{\langle r \rangle}(t)$ converge when $r$ goes to infinity. In particular,
    this holds if $\sum_{n \geq 0} a_n t^n$ is the Hilbert series of a standard
    graded $k$-algebra $A$. If in addition $A$ is Cohen-Macaulay then the coefficients
    of $h^{\langle r \rangle}(t)$ are monotonely increasing with $r$.
    If $A$ is the Stanley-Reisner ring of
    a simplicial complex $\Delta$ then this relates to the $r$th edgewise subdivision
    of $\Delta$ which in turn allows some corollaries on the behavior of the
    respective $f$-vectors. 
  \end{abstract}

  \maketitle

  \section{Introduction and Statement of Results} \label{introduction}

    In this paper we study for a rational formal power series of the form
    $$f(t) := \displaystyle{\sum_{n \geq 0} a_n t^n = \frac{h(t)}{(1-t)^{d}}},~ a_n \in \CC \mathrm{~for~} n \geq 0,$$
 the transformation of the numerator polynomial when passing for
    some number $r \geq 1$ to the generating function 
    $f^{\langle r \rangle}(t) := \sum_{n \geq 0} a_{rn} t^n = \frac{h^{\langle r \rangle}(t)}{(1-t)^{d}}$.

    We are motivated by the following facts from commutative algebra.
    Let $A = \bigoplus_{n \geq 0} A_n$ be a standard graded $k$-algebra; that is
    $A$ is finitely generated in degree $1$ and $A_0 = k$. 
    The Hilbert-Serre Theorem \cite[see Chap. 10.4]{Eisenbud} asserts that its Hilbert-series 
    $\Hilb(A,t) = \sum_{n \geq 0} \dim_k A_n \,  t^n$
    is a rational function of the form $\Hilb(A,t) = \frac{h(t)}{(1-t)^{d}}$ for
    some polynomial $h(t)$ such that $h(1) \geq 1$ and $d$ the Krull dimension of $A$.
    The $r$th Veronese algebra of $A$ is the standard graded $k$-algebra 
    $A^{\langle r \rangle} := \bigoplus_{n \geq 0} A_{nr}$ with Hilbert-series
    $\Hilb(A^{\langle r \rangle},t) = \sum_{n \geq 0} \dim_k A_{rn} t^n$ $=$ Hilb$(A,t)^{\langle r \rangle}$.
    Veronese algebras are well studied objects in commutative algebra and algebraic geometry.
    In particular, the limiting behavior of algebraic properties of 
    $A^{\langle r \rangle}$ for large $r$ has been considered in \cite{Backelin}, \cite{EisenbudReevesTotaro} and more
    generally in \cite{ConcaHerzogTrungValla}. Indeed as will be seen
    later the results from \cite{Backelin} and \cite{EisenbudReevesTotaro} relate to our own results.

    Most results presented in this paper are very much in the spirit of results from \cite{BrentiWelker} 
    on the behavior of the $h$-vector and $h$-polynomial of barycentrically subdivided simplicial complexes.
    But even though the formulations of the theorems appear to be very similar the proofs are almost disjoint
    except for the use of Lemma \ref{technicallemma} in the proof of Theorem \ref{limittheorem}. 

    Before we can state the first main result we have to define the following numbers.
    Here and in the sequel $\NN$ will denote the nonnegative integers and $\PP$ the
    strictly positive integers. 
    For $d,r,i \in \NN$ let
    \[ 
       C(r,d,i) := \Big| \big\{~ (a_1, \ldots , a_d) \in \NN^d ~\big|~
                                            \begin{array}{c}  a_1 + \cdots + a_d = i \\
                                                              a_j \leq r {\mathrm ~for~} 1 \leq j \leq d
                                            \end{array} \big\} 
       \Big|,
    \]
    for $d \geq 1$ and $C(r,0,i) = \delta_{0,i}$.
    Note that it is easy to see that
    \[
       C(r,d,i) = \sum_{\{ \lambda  \subseteq  (r^d): \; |\lambda | =i \}}
                       \left( ^{^{\scriptstyle \hspace{1,6cm} d}}_{_{\scriptstyle m_{1} (\lambda ), \ldots ,
                       m_{r}(\lambda ),d-l(\lambda )}} \right)   ,
    \]
where, for a partition $\lambda$, $\ell (\lambda )$ denotes the number of parts of $\lambda$
and $m_{i}(\lambda )$ the number of parts of $\lambda$ that are equal to $i$.

    \begin{theorem} \label{positivelinear}
      Let $(a_n)_{n \geq 0}$ be a sequence of complex numbers such that
      for some $s,d \geq 0$ its generating series $f(t) := \sum_{n \geq 0} a_nt^n$ satisfies
      $$f(t) = \frac{h_0 + \cdots + h_st^s}{(1-t)^d}.$$
      Then for any $r \in \PP$ we have
      $$f^{\langle r \rangle}(t)  = \sum_{n \geq 0} a_{nr} t^n = \frac{h_0^{\langle r \rangle} +
      \cdots + h_m^{\langle r \rangle}t^m}{(1-t)^{d}},$$
     where m:=max(s,d) and
        $$h_i^{\langle r \rangle} = \sum_{j = 0}^s C(r-1,d,ir-j)\, h_j ,$$
	for $i=0, \ldots ,m$.
    \end{theorem}

    The following is a simple reformulation of Theorem \ref{positivelinear}
    in the case of Hilbert-series of standard graded $k$-algebras.

    \begin{corollary} \label{algebrapositivelinear}
       Let $A$ be a standard graded $k$-algebra of dimension $d$
       with Hilbert-series
       $$\Hilb(A,t) = \frac{h_0 + \cdots + h_st^s}{(1-t)^{d}}.$$
       Then for any $r \in \PP$ we have
       $$\Hilb(A^{\langle r \rangle},t) = \frac{h_0^{\langle r \rangle} +
       \cdots + h_m^{\langle r \rangle}t^m}{(1-t)^{d}},$$
       where m:=max(s,d) and
        $$h_i^{\langle r \rangle} = \sum_{j = 0}^s C(r-1,d,ir-j)\, h_j ,$$
	for $i= 0, \ldots , m$.
    \end{corollary}

    As a consequence it follows that the first $d+1$ entries of the
     $h$-vector $h(A) = (h_0, \ldots, h_s)$
    of a standard graded algebra grow weakly when taking Veronese subalgebras in
    case $h_i \geq 0$ for $0 \leq i \leq s$. Note that this condition for example
    is satisfied if $A$ is Cohen-Macaulay (see for example \cite[Proposition 4.3.1]{BrunsHerzog}).

    \begin{corollary} \label{growth}
       Let $A$ be a standard graded $k$-algebra of dimension $d$
       with Hilbert-series
       $$\Hilb(A,t) = \frac{h_0 + \cdots + h_st^s}{(1-t)^{d}}$$
       such that $h_i \geq 0$ for $i=0,\ldots,s$.
       Then for any $r \geq 1$ the Hilbert-series
       $$\Hilb(A^{\langle r \rangle},t) = \frac{h_0^{\langle r \rangle} +
       \cdots + h_m^{\langle r \rangle}t^m}{(1-t)^d}$$ of the $r$th Veronese algebra of $A$
       satisfies
       $h_i^{\langle r \rangle} \geq h_i$ for $0 \leq i \leq d$.
       Moreover, if $r \geq d$ then $h_i^{\langle r \rangle} > h_i$
       for $1 \leq i \leq d-1$. 
       In particular, all conclusions hold if $A$ is Cohen-Macaulay.
    \end{corollary}
    \begin{proof}
      By Corollary \ref{algebrapositivelinear} it follows that 
      $h_i^{\langle r \rangle} = \sum_{j = 0}^s C(r-1,d,ir-j)\, h_j$. Clearly,
      $C(r-1,d,ir-j) \geq 0$ for all $r,d,i,j$. Moreover, for $0 \leq i \leq d$ we have
      $C(r-1,d,ir-i) \geq 1$ since the sum
      $$i(r-1) = \underbrace{(r-1)+ \cdots + (r-1)}_{i\mathrm{~times}} +
                 \underbrace{0 + \cdots + 0}_{(d-i) \mathrm{~times}}$$ is
      clearly among the sums counted by $C(r-1,d,i(r-1))$.
      This implies $h_i^{\langle r \rangle} \geq h_i$ for $0 \leq i \leq d$. 
      It is well known that for a standard graded algebra we have $h_0 = 1$.
      Now if $r \geq d$ then for $1 \leq i \leq d-1$ we have $ir \leq (d-1)r
      \leq d(r-1)$.
       Thus there is at least one sum representation
      of $ir$ with $d$ summands $\leq r-1$. Hence $C(r-1,d,ir) \geq 1$ and
      therefore $C(r-1,d,ir-0) h_0 \geq 1$ which then implies
      $h_i^{\langle r \rangle} \geq h_i+1$ for $1 \leq i \leq d-1$.
    \end{proof} 
    
    Note that for the Hilbert-series 
    $\Hilb(A,t) = \frac{h_0 + \cdots + h_dt^d}{(1-t)^{d}}$ of a standard
    graded algebra $A$ it is well known that $h_0^{\langle r \rangle} = h_0 = 1$ and
    $h_d^{\langle r \rangle} = h_d$. Of course these identities also follow from
    Theorem \ref{positivelinear}.
 
    \begin{theorem} \label{limittheorem}
       For any $d \geq 2$ there are strictly negative real numbers
       $\alpha_1 \ldots, \alpha_{d-2}$ such that for any
       sequence $(a_n)_{n \geq 0}$ of real numbers such that $a_0 = 1$ and $a_n \geq 0$ for large $n$ whose
       generating series $f(t) = \sum_{n \geq 0} a_nt^n = \frac{h(t)}{(1-t)^d}$
       for some polynomial $h(t)$ with $h(1) \neq 0$
       there is a number $R > 0$ and sequences of complex numbers $(\beta_r^{(i)})_{r \geq 1}$, $1 \leq i \leq d$,
       such that :
       \begin{itemize}
          \item[(i)] $\beta_r^{(i)}$ is real for $r > R$ and $1 \leq i \leq d$ and strictly negative
                     for $r > R$ and $1 \leq i \leq d-1$.
          \item[(ii)] $\beta_r^{(i)} \rightarrow  \alpha_i$ for $r \rightarrow \infty$ and $1 \leq i \leq d-2$.
          \item[(iii)] $\beta_r^{(d-1)} \rightarrow -\infty$ for $r \rightarrow \infty$.
          \item[(iv)] $\beta_r^{(d)} \rightarrow 0$ for $r \rightarrow \infty$.
          \item[(v)] $h_0^{\langle r \rangle} + \cdots + h_d^{\langle r \rangle}t^d =
                       \prod_{i=1}^d (1-\beta_r^{(i)}t)$.
       \end{itemize}
    \end{theorem}

   \begin{corollary} \label{algebralimitI}
      For any $d \geq 2$ there are strictly negative real numbers
      $\alpha_1 \ldots, \alpha_{d-2}$ such that for any
      standard graded $k$-algebra of dimension $d$
      with Hilbert-series
      $$\Hilb(A,t) = \frac{h_0 + \cdots + h_st^s}{(1-t)^{d}}$$
      there are $R > 0$ and sequences of complex numbers $(\beta_r^{(i)})_{r \geq 1}$, $1 \leq i \leq d$,
      such that~:
      \begin{itemize}
         \item[(i)] $\beta_r^{(i)}$ is real for $r > R$ and $1 \leq i \leq d$ and strictly negative
                    for $r > R$ and $1 \leq i \leq d-1$.
         \item[(ii)] $\beta_r^{(i)} \rightarrow  \alpha_i$ for $r \rightarrow \infty$ and $1 \leq i \leq d-2$.
         \item[(iii)] $\beta_r^{(d-1)} \rightarrow -\infty$ for $r \rightarrow \infty$.
         \item[(iv)] $\beta_r^{(d)} \rightarrow 0$ for $r \rightarrow \infty$.
         \item[(v)] $h_0^{\langle r \rangle} + \cdots + h_d^{\langle r \rangle}t^d =
                       \prod_{i=1}^d (1-\beta_r^{(i)}t)$, for $r>R$.
       \end{itemize}
    \end{corollary}
 
    In \cite{Backelin} and \cite{EisenbudReevesTotaro} it is shown that for a standard graded $k$-algebra $A$ and
    $r$  large enough the $r$th Veronese $A^{\langle r \rangle}$ is Koszul (see \cite[p. 450]{Eisenbud}).
    In turn it is known (see for example
    \cite{ReinerWelker}) that this
    property implies that the numerator polynomial of the Hilbert-series has at least one real root. Thus in
    some sense the previous and the following corollary are inspired by this algebraic limiting results. 

    \begin{corollary} \label{algebralimitII} 
       Let $A$ be a standard graded $k$-algebra of dimension $d \geq 1$
       with Hilbert-series
       $$\Hilb(A,t) = \frac{h_0 + \cdots + h_st^s}{(1-t)^{d}}.$$
       Then there is $R>0$ such that for any $r > R$ we have:
       \begin{itemize}
         \item[(i)] $h_i^{\langle r \rangle} \geq 1$, $0 \leq i \leq d-1$,
         \item[(ii)] $h_i^{\langle r \rangle} = 0$ for $i \geq d+1$
         \item[(iii)] $h^{\langle r \rangle}(t)$ has only real zeros.
       \end{itemize}
       In particular, for $r > R$ the sequence $(h_0^{\langle r \rangle}, \ldots, h_d^{\langle r \rangle})$
       is log-concave and unimodal. 
    \end{corollary}

  \section{Proof of Theorem \ref{positivelinear}} \label{proofpositivelinear}

     \begin{proof}[Proof of Theorem \ref{positivelinear}]
        Set $h(t) := \sum_{i=0}^s h_i\, t^{i}$
        so that $f(t) = \frac{h(t)}{(1-t)^{d}}$.
        Let $\rho \in \CC$ be a primitive $r$-th root of unity. Then $\rho$ and all its powers
        $\rho^j$ for any $j \in \NN$ such that $j \not \equiv 0 \pmod{r}$ satisfy
        $\sum_{i=0}^{r-1}(\rho^j)^{i}=0$. Hence:
        \begin{eqnarray*}
          \sum_{n\geq 0} a_{rn} t^{rn}             & = & \frac{1}{r} \, \sum_{i=0}^{r-1}f(\rho^i \, t) \\
                                                   & = & 
                            \frac{1}{r} \, \sum_{i=0}^{r-1} \frac{h(\rho^i \, t)}{(1-\rho^i \, t)^d} \\
                                                   & = &  
                            \frac{{\displaystyle \sum_{i=0}^{r-1}} \left( \frac{1-t^r}{1-\rho^i \, t} \right)^d
                                                                        \, h(\rho^i \, t)}{r(1-t^r)^d} \\
                                                   & = &  
                            \frac{{\displaystyle \sum_{i=0}^{r-1}} \left( \frac{1-(\rho^i\, t)^r}{1-\rho^i \, t}
                                                           \right)^{d} \, h(\rho^i \, t)}{r \, (1-t^r)^{d}}
        \end{eqnarray*}
        But
        \begin{eqnarray*}
          \frac{1}{r} \, \sum_{\ell =0}^{r-1} \left( \frac{1-(\rho^{\ell}t)^{r}}{1-\rho
	  ^{\ell}t} \right)^{d} \, h(\rho^{\ell } t)
                                                  & = & \frac{1}{r} \sum_{\ell=0}^{r-1}
                                        \Big( \big(1+\rho^\ell\, t+ \cdots + (\rho^\ell \, t)^{r-1}\big)^{d} h(\rho^\ell \, t)\Big) \\
                                                  & = & \frac{1}{r} \, \sum_{\ell=0}^{r-1} \left( \sum_{i \geq 0}
                                        C(r-1,d,i) (\rho^\ell \, t)^{i} \right) \, h(\rho^\ell \, t) \\
                                                  & = & \frac{1}{r} \sum_{i \geq 0} C(r-1,d,i) \left( \sum_{\ell=0}^{r-1}
                                        (\rho^\ell \, t)^i  \sum_{j=0}^s h_j \, (\rho^\ell \, t)^j \right) \\
                                                  & = & \frac{1}{r} \sum_{j=0}^s  \sum_{i \geq 0}
                                        C(r-1,d,i-j) \, h_j \, \sum_{\ell=0}^{r-1} (\rho^\ell \, t)^i \\
                                                  & = & \sum_{i \geq 0} \left(
                                        \sum_{j = 0}^s C(r-1,d,ir-j)\, h_j \right) t^{ri}
        \end{eqnarray*}
        and the result follows since if $i>s \geq d$ then $ir -j>sr-s \geq d(r-1)$
	for all $0 \leq j \leq s$ while if $i>d>s$ then $ir-j>dr-d$ for all $0\leq j \leq s$.
    \end{proof}

    We now analyze the transformation described in
    Theorem \ref{positivelinear} more closely. 
    Consider the vectorspace $\Rat_{d}$ of all rational functions
    $\frac{h(t)}{(1-t)^{d}}$ for polynomials $h(t)$ of degree $\leq d$.
    We consider two bases of $\Rat_d$. First the basis
    $\Bas_d^1$ consisting of all $\frac{t^i}{(1-t)^{d}}$ for $0 \leq i \leq d$.
    For the second basis we recall the definition of an Eulerian polynomial.
    For a number $i \geq 1$ we define $A_i(t) = \sum_{\sigma \in S_i} t^{\des(\sigma)+1}$,
    where $\des(\sigma)$ is the number of descents of the permutation $\sigma$.
    If we also set $A_0(t) = A_{-1}(t) = 1$ then the set $\Bas_d^2$ consisting of all
    $\frac{A_{i-1}(t)(1-t)^{d-i}}{(1-t)^{d}}$, $0 \leq i \leq d$, is a second
    basis of $\Rat_d$. Using the fact that for $1 \leq i$ one has $A_{i}(1)
    \neq 0$ one easily checks that indeed $\Bas_d^2$
    is a basis of $\Rat_{d}$.

    For a fixed  $r \geq 1$ let $\Phi_r : \Rat_d \rightarrow \Rat_d$ be the map that
    sends $f(t)$ to $f^{\langle r \rangle}  (t)$. Note that in the basis $\Bas_d^2$ the
    map $\Phi_r$ sends $\sum_{n \geq 0} n^{i} t^n = \frac{A_i(t)}{(1-t)^{i+1}}$
    for $0 \leq i < d$ to

    \begin{eqnarray*} 
      \Phi_r(\sum_{n \geq 0} n^{i} t^n) & = & \sum_{n \geq 0} (rn)^{i} t^n  \\
                                      & = & r^{i} \sum_{n \geq 0} n^{i} t^n \\
                                      & = & r^{i} \frac{A_i(t)}{(1-t)^{i+1}}
    \end{eqnarray*}
\noindent
    while $\Phi_r(A_{-1}(t))=A_{-1}(t)$.
    In particular, this confirms that $\Phi_r$ is a map from $\Rat_d$ to $\Rat_d$.
    Moreover, since $\Phi_r$ is easily seen to be linear it follows that
    in the basis $\Bas_d^2$ the map $\Phi_r$ is represented by the $(d+1) \times (d+1)$
    diagonal matrix $\diag(1,1,r,\ldots, r^{d-1})$.
    The preceding arguments and Theorem \ref{positivelinear}
    imply the following lemma.

    \begin{lemma} \label{diagonal}
       Let $\fC_{d,r} = (C(r-1,d,ir-j))_{0 \leq i ,j \leq d}$. Then $\fC_{d,r}$ is
       the matrix representing the linear transformation $\Phi_r$ with respect
       to the basis $\Bas_d^1$. In particular, $\fC_{d,r}$ is diagonizable with eigenvalues
       $1$ of multiplicity two and $r,\ldots, r^{d-1}$ of multiplicity one.
    \end{lemma}

    Indeed we can give a factorization of $\fC_{d,r}$ which also clarifies its
    eigenspaces. Let $\fL_d = (l_{ij})_{0 \leq i,j \leq d}$
    be the $(d+1) \times (d+1)$ matrix with entries $l_{ij}$ defined by
    $A_{i-1}(t) (1-t)^{d-i} = \sum_{j=0}^{d} l_{j,i} t^j$.

    \begin{lemma} \label{eigenvecs}
       For any $d,r \geq 1$ we have
       $$\fC_{d,r} = \fL_d \cdot  \diag(1,1,r,\ldots, r^{d-1}) \cdot \fL_d^{-1}.$$
       The vector $\ell_i = (l_{0i}, \ldots, l_{di})^t$ is an eigenvector
       of $\fC_{d,r}$ for the eigenvalue  $r^{i-1}$ for $1 \leq i \leq d$.
       Moreover:
       \begin{itemize}
         \item[(i)] $l_{d,i} = 0$ for $1 \leq i \leq d$ and $l_{d,0} =(-1)^d$.
         \item[(ii)] $l_{0,i} = 0$ for $2 \leq i \leq d$ and $l_{0,0} = l_{0,1} = 1$.
         \item[(iii)]  $l_{j,d} \geq 1$, $1 \leq j \leq d-1$.
       \end{itemize}
    \end{lemma}
    \begin{proof}
       The matrix $\fL_d$ describes the base change from $\Bas_d^2$ to $\Bas_d^1$.
       Since by the arguments preceding Lemma \ref{diagonal} $\diag(1,1,r,\ldots, r^{d-1})$
       is the matrix of $\Phi_r$ with respect to the basis $\Bas_d^2$ and
       since by Theorem \ref{positivelinear} $\fC_{d,r}$ is the matrix of $\Phi_r$
       with respect to the basis $\Bas_d^1$ it follows that $\fC_{d,r} =
       \fL_d \diag(1,1,r,\ldots, r^{d-1}) \fL_d^{-1}.$
       The preimage of the $i$th unit vector under $\fL_d^{-1}$ is $\ell_i$.
       Since the $i$th unit vector is an eigenvector for $r^{i-1}$ of
       $\diag(1,1,r,\ldots, r^{d-1})$
       it follows by the first part of the lemma that $\ell_i$ is an eigenvector
       of $\fC_{d,r}$ for the eigenvalue $r^{i-1}$, $1 \leq i \leq d$.

       Assertions (i) - (iii) are immediate consequences of the definitions.
    \end{proof}

     The preceding lemma  implies the following combinatorial identity for Eulerian
     numbers.
     For $d \geq 1$ and $0 \leq i \leq d$ let $A(d,i) = |\{ \sigma \in S_d \, : \,
     \des(\sigma)=i-1 \}|$.

     \begin{proposition} \label{combidentity}
     Let $d,r \geq 1$. Then
        $$\sum_{j=0}^d C(r-1,d+1,ir-j) A(d,j) = r^{d} A(d,i)$$
     for $i=0,\ldots,d$. In particular
     $$\sum_{j=0}^d C(r-1,d+1,r-j) A(d,j) = r^{d}.$$
     \end{proposition}

     Clearly, Proposition \ref{combidentity} asks for a combinatorial proof.

     \section{Proof of Theorem \ref{limittheorem}} \label{prooflimiththeorem}

     Before we come to the proof of Theorem \ref{limittheorem} we need some preparatory lemmas.

     \begin{lemma} \label{cutdegree} Let $f(t) = \frac{h_0+ \cdots + h_st^s}{(1-t)^d}$ for some $d,s \geq 0$.
        Write $f(t) = p(t) + \frac{h_1(t)}{(1-t)^d}$ for polynomials $p(t)$ and $h_1(t)$ where $h_1(t)$ is of degree $\leq d-1$.
        Then for any $r \geq s-d+1$ we have $$f(t)^{\langle r \rangle} = \frac{h^{\langle r \rangle}(t)}{(1-t)^d} = f_1(t)^{\langle r \rangle}$$
        for some polynomial $h^{\langle r \rangle}(t)$ of degree $\leq d$
        and $f_1(t) = \frac{p(0)(1-t)^d +h_1(t)}{(1-t)^d}$.
        Moreover, $h_0 = h^{\langle r \rangle}(0)$ for all $r$.
     \end{lemma}
     \begin{proof}
If $s \leq d$ then the result follows immediately from Theorem \ref{positivelinear}  so assume
$s>d$. Then
	$f(t)=
       p_0+p_1t+\cdots + p_{s-d}t^{s-d} + \frac{h_1(t)}{(1-t)^d}$ for a polynomial $h_1(t)$ of degree $\leq d-1$.
     Hence $f(t)^{\langle r \rangle} = p_0 + \frac{h_1(t)^{\langle r \rangle }}{(1-t)^d} = f_1(t)^{\langle r \rangle}$ for $r \geq s-d+1$.
       Since by  Theorem \ref{positivelinear} $h_1(t)^{\langle r\rangle}$ is a polynomial of degree $\leq d$ the assertion follows.
 
       The last equality follows by evaluating the generating series at 
       $t = 0$. 
     \end{proof}

     \begin{remark}
       Let $f(t) = \sum_{n=0}^\infty a_nt^n = \frac{h(t)}{(1-t)^d}$.
       If $d = 0$ then for sufficiently large $r$ we have $f(t)^{\langle r \rangle}=a_{0}$.
       If $d = 1$ then for sufficiently large $r$
       we have $f(t)^{\langle r \rangle}  = \frac{h_0+h_1t}{1-t}$ with $h_{0},h_{1}$ independent of $r$. In particular, the
       numerator polynomial of $f(t)^{\langle r \rangle}$ is real rooted with at most one root which is
       independent of $r$.
     \end{remark}

     We recall the following lemma from \cite{BrentiWelker}.

     \begin{lemma}[Lemma 4.9 \cite{BrentiWelker}] \label{technicallemma}
        Let $(g_n(t))_{n \geq 0}$ be a sequence of real polynomials of degree
        $d-2$, $f(t)$ another real polynomial of degree $d-2$
        and $\rho > 1$, $h_d$ real numbers such that:
        \begin{itemize}
          \item[$\triangleright$] $\lim_{n \rightarrow \infty} g_n(t)/\rho^n =
             0$, where the limit is taken in $\RR^{d-1}$.
          \item[$\triangleright$] All the roots of the polynomial $f(t)$ are strictly
             negative and all coefficients of $f(t)$ are strictly positive.
        \end{itemize}
           Then there are real numbers $\alpha_i$, $1 \leq i \leq d-2$ and
           sequences $(\beta_i^{(n)})_{n \geq 0}$, $1 \leq i \leq d$ of complex numbers such that:
        \begin{itemize}
           \item[(i)] $\beta_i^{(n)}$, $1 \leq i \leq d$, are real for $n$ sufficiently large.
           \item[(ii)] $\lim_{n \rightarrow \infty} \beta_i^{(n)} = \alpha_i$, $1 \leq i \leq d-2$.
           \item[(iii)] $\lim_{n \rightarrow \infty} \beta_{d-1}^{(n)} = 0$.
           \item[(iv)] $\lim_{n \rightarrow \infty} \beta_{d}^{(n)} = - \infty$.
           \item[(v)] $\prod_{i = 0}^{d-1} (t-\beta_i^{(n)}) = h_d + tg_n(t) + \rho^n tf(t) +t^d$.
         \end{itemize}
     \end{lemma}

     \begin{proof}[Proof of Theorem \ref{limittheorem}]
       By Lemma \ref{cutdegree} we may assume that $h(t)$ is of degree $\leq d$.
       Then Theorem \ref{positivelinear} implies that for $f^{\langle r \rangle}(t) = \frac{h^{\langle r \rangle}(t)}{(1-t)^d}$ the
       polynomial $h^{\langle r \rangle}(t)$ is again of degree $\leq d$.

       Let $\ell_i = (l_{0i}, \ldots, l_{d,i})^{t}$, $0 \leq i \leq d$ be the eigenvectors of $\fC_{d,r}$ as in
       Lemma \ref{eigenvecs}. 
       Let $\ell_i(t) = \frac{\sum_{j=0}^{d} l_{j,i}t^j}{(1-t)^{d}}$, $0 \leq i \leq d$ so $\Bas_d^2 = \{ \ell_0(t),
       \ldots, \ell_d(t) \}$. Let $f(t) =  \alpha_0 \ell_0(t) + \cdots + \alpha_{d} \ell_{d}(t)$ be the expansion of $f(t)$ in the
       basis $\Bas_d^2$. Then by Lemma \ref{eigenvecs} we have 
       $f(t)^{\langle r \rangle} = \alpha_0 \ell_0(t) + \alpha_1 \ell_1(t) + \alpha_2 r \ell_2(t) + \cdots + \alpha_{d} r^{d-1} \ell_{d}(t)$.
       Let
       \[ g_{r}(t) \stackrel{\rm def}{=} t^{d} \, h^{\langle r \rangle }
       \, \left( \frac{1}{t} \right) -h_{d}-t^{d}-r^{d-1} \, \alpha _{d} \, t^{d}
       \, \tilde{\ell}_{d} \left( \frac{1}{t} \right) \]
       where $\tilde{\ell}_{d}(t) \stackrel{\rm def}{=}(1-t)^{d} \, \ell _{d}(t)$.
       Then deg$(g_{r}) \leq d-1$, $g_{r}(0) = 0$ and $\lim_{r \rightarrow \infty} g_r(t) /r^{d-1} = 0$.
       Now $\tilde{\ell}_d(t) = A_{d-1} (t)$ is the $(d-1)$st Eulerian polynomial.
        Thus $\ell_d(t) = t^d \ell_d(1/t)$. Hence  if we set $\tilde{g}_{r}(t)
	\stackrel{\rm def}{=}g_{r}(t)/t$, $\tilde{f}(t) \stackrel{\rm def}{=}
	A_{d-1}(t)/t$
       then $\tilde{f}(t)$ is a polynomial of degree $d-2$, which is real rooted by \cite[p. 292, Ex. 3]{Comtet}
       and has strictly positive coefficients except for the constant, and all the
       hypotheses of Lemma \ref{technicallemma} are satisfied.

       Hence Lemma \ref{technicallemma} becomes applicable and the result follows by passing to the reciprocal polynomials.
     \end{proof}

     It remains to provide  proofs of  Corollaries \ref{algebralimitI} and \ref{algebralimitII}.
   Corollary \ref{algebralimitI} is just a reformulation of Theorem \ref{limittheorem} for
     Hilbert-series. So there is nothing to prove.
     But note that in general passing to Veronese subrings $A^{\langle r \rangle}$ for small $r$ does not
     suffice to guarantee that $h^{\langle r \rangle}(t)$ is real rooted, even in case $h(t)$ is of
     degree $d-1$ and has strictly positive coefficients. 
     For example, $f(t) = \frac{1+t+t^2+t^3+t^4}{(1-t)^5}$ is the
     Hilbert-series of the Stanley-Reisner ring of the boundary complex of the $4$-simplex. 
     However $f^{\langle 2 \rangle} (t) = \frac{1+16t+31t^2+26t^3+6t^4}{(1-t)^5}$ whose numerator polynomial has
     only two real roots.

     \begin{proof}[Proof of \ref{algebralimitII}]
       If $d=1$ then (iii)  follows from (ii) which follows
       immediately from Lemma \ref{cutdegree}.

       Let $d \geq 2$. Assertions (ii) and (iii) follow immediately from Lemma
       \ref{cutdegree} and Corollary \ref{algebralimitI}.
       From the proof of Theorem \ref{limittheorem}  we  recall that
       the numerator polynomial of $\Hilb(A^{\langle r \rangle},t)$ can be written as
       \begin{equation}
       \label{1.6.1}
       h_dt^d + t^{d}g_r\left( \frac{1}{t}\right) + \alpha_{d} r^{d-1} A_{d-1} (t)  + 1
       \end{equation}
       for polynomials $g_r(t)$ and the Eulerian polynomial $A_{d-1}(t)$.
       Moreover we have $\lim_{r \rightarrow \infty} g_r(t)/r^{d-1} = 0$
       and $A_{d-1}(t)$ has strictly positive coefficients except for the constant term which
       is $0$. But this implies
       that for large enough $r$ the polynomial $1/r^{d-1} (t^{d}g_r(1/t)+r^{d-1} A_{d-1}(t))$
       has strictly positive coefficients. This then implies that for large enough
       $r$ the polynomial in (\ref{1.6.1}) has strictly positive coefficients
       except for possibly $h_d$, which proves (i).
     \end{proof}

     \section{Veronese of Stanley-Reisner rings} \label{srrings}
  
     For a simplicial complex $\Delta$ over ground set $\Omega$ and a field $k$ we denote by
     $k[\Delta]$ its Stanley-Reisner ring. Recall that $k[\Delta] = k[x_\omega~:~\omega \in \Omega] / I_\Delta$,
     where $I_\Delta$ is the ideal generated by the monomials $\prod_{\omega \in A} x_\omega$ for $A \not\in
     \Delta$. Assume that $\Delta$ is $(d-1)$-dimensional. We denote
     by $f(\Delta) = (f_{-1}, \ldots, f_{d-1})$ the $f$-vector of $\Delta$; that is $f_i$ is the number of $i$-dimensional faces
     of $\Delta$.
     Then it is well known that :
     \begin{eqnarray} \label{aboveequation} 
        \Hilb(k[\Delta],t) & = &\frac{h_0 + \cdots + h_dt^d}{(1-t)^{d}} \\
        \nonumber                  & = & \frac{\displaystyle{\sum_{i=0}^d f_{i-1} t^i(1-t)^{d-i}}}{(1-t)^{d}}
     \end{eqnarray} 
     The $r$th Veronese of $k[\Delta]$ is a Stanley-Reisner  ring only in extremal cases, but still it has
     turned out to be fruitful and meaningful to look for a simplicial complex $\Delta^{\langle r \rangle }$ such that
     \begin{eqnarray} \label{defineequation} 
        \Hilb(k[\Delta]^{\langle r\rangle },t) & = & \Hilb(k[\Delta^{\langle r \rangle }],t).
     \end{eqnarray}

     In \cite{BrunRoemer} based on earlier ideas by Sturmfels \cite{Sturmfels}, Brun
     and R\"{o}mer consider
     the following situation. Set $S = k[x_1, \ldots, x_n]$ the polynomial ring and let $I_\Delta$ be the
     Stanley-Reisner ideal of a simplicial complex $\Delta$ on ground set $[n]$.
     Then the $r$-th Veronese $(S/I_\Delta)^{\langle r \rangle}$ can be described as a quotient of the polynomial ring
     $S(r)$ in the variables $x_{i_1, \ldots, i_n}$ indexed by numbers $0 \leq i_1, \ldots, i_n$ such that
     $i_1 + \cdots + i_n = r$. If $I(r)$ is the ideal in $S(r)$ such that $(S/I_\Delta)^{\langle r \rangle} = S(r)/I(r)$
     then Brun and R\"omer \cite[Section 6]{BrunRoemer} describe a initial ideal of $I(r)$ which is the Stanley-Reisner
     ideal of a simplicial complex $\Delta(r)$ on
     vertex set $\Omega_r = \{ (i_1,\ldots, i_n) \in \NN^n~:~i_1 + \cdots + i_n = r\}$. By basic facts about
     initial ideals it follows that this $\Delta(r)$ fulfills (\ref{defineequation}). The simplicial complex
     $\Delta(r)$ turns out to be realizable as a subdivision
     of $\Delta$. This subdivision is called $r$-th edgewise subdivision and as outlined in \cite{BrunRoemer} has a
     long history in algebraic topology (see e.g. \cite{Freudenthal}, \cite{Grayson}) and a shorter one  in discrete
     geometry (see e.g.\cite{EdelsbrunnerGrayson}).

     Before we can describe edgewise subdivision we need some
     technical preparations. Consider $\RR^n$ together with its standard unit
     basis vectors $\ee_1, \ldots, \ee_n$. By the obvious identification we can consider
      $\Delta$ as a simplicial
     complex over the ground set $\{ \ee_1, \ldots, \ee_n \} = \Omega_1$.
     Note that for $r \geq 1$ the elements of $\Omega_r$ are the points with integer coordinates in the
     the $r$-th dilation of the simplex spanned by $\Omega_1$. 
     For a vector $\aa = (a_1,\ldots, a_n) \in \RR^n$ its support $\supp(\aa)$ is the set $\{ i ~:~a_i \neq 0 \}$
     of indices of non-zero coordinates. For $i \in [n]$ set $\uu_i = \ee_i + \ee_{i+1} + \cdots + \ee_n$ and for
     $\aa = (a_1,\ldots, a_n) \in \RR^n$ set $\iota(\aa) := \sum_{i=1}^n a_i \uu_i$.
 
     The $r$-th edgewise subdivision of $\Delta$ is the simplicial complex $\Delta (r)$
      on ground set $\Omega_r$
     such that $A \subseteq \Omega_r$ is a simplex of $\Delta (r)$ if and only if
     \begin{itemize}
        \item[(ESD1)] $\displaystyle{\bigcup_{\vv \in A}} \supp(\vv) \in \Delta$.
        \item[(ESD2)] For all $\vv,\vv' \in A$ either $\iota(\vv-\vv') \in \{0,1\}^n$ or $\iota(\vv'-\vv) \in \{0,1\}^n$.
     \end{itemize}

     Now the result by Brun and R\"omer \cite{BrunRoemer} states.

     \begin{proposition}[Proposition 6.4 in \cite{BrunRoemer}] \label{brunroemerprop}
        Let $\Delta$ be a simplicial complex on ground set $[n]$ and  $I(r)$ be such that
        $(S/I_\Delta)^{\langle r \rangle} = S(r)/I(r)$. Then there is a term order for which
        $I_{\Delta(r)}$ is the initial ideal of $I(r)$.
     \end{proposition}

     In particular, $\Delta^{\langle r\rangle } := \Delta(r)$ satisfies Equation (\ref{defineequation}).
     Clearly, a simplicial complex $\Delta^{\langle r\rangle }$ satisfying
     Equation (\ref{defineequation}) is not uniquely defined. But $\Delta(r)$ appears to be a natural choice.

     First we want to study the enumerative properties of $\Delta(r)$. A partial analysis can also be found in
     \cite{EdelsbrunnerGrayson} but there the final formulas appear as alternating sums which is not
     fully satisfactory from the point of view of Enumerative Combinatorics.
     For enumerative purposes Equation (\ref{aboveequation}) suggests to study a third basis of $\Rat_d$. We denote by
     $\Bas_d^3$ the set of rational functions $\frac{t^i(1-t)^{d-i}}{(1-t)^{d}}$, $0 \leq i \leq d$. Indeed  $\Bas_d^3$
     will be crucial in the proof of the following proposition.

     \begin{proposition}
        Let $\Delta$ be a simplicial complex of dimension $d-1$ with $f$-vector
        $f(\Delta) = (f_{-1} , \ldots, f_{d-1})$ and $r \in \PP$.
        If $\Delta^{\langle r \rangle}$ is a simplicial complex such that
        $\Hilb(k[\Delta]^{\langle r\rangle },t) = \Hilb(k[\Delta^{\langle r\rangle }],t)$ then its
        $f$-vector $f(\Delta^{\langle r\rangle }) = (f_{-1}^{\langle r\rangle } , \ldots, f_{d-1}^{\langle r\rangle })$
        satisfies
        $$f_{i-1}^{\langle r \rangle} = \sum_{\ell = i}^d \sum_{{j_1+ \cdots + j_i = \ell} \atop {j_1, \ldots, j_i \geq 1}} {r-1 \choose j_1-1} {r \choose j_2} \cdots {r \choose j_i} f_{\ell-1} ,$$
	for $0 \leq i \leq d$.
     \end{proposition}
     \begin{proof}
        We denote by $v_i(t) = \frac{t^i(1-t)^{d-i}}{(1-t)^{d}}$,
        $0 \leq i \leq d$, the elements of $\Bas_d^3$. Then there are numbers $a_{i,\ell}
	^{(r)}$ ($0\leq i, \ell \leq d$) such that
        \begin{eqnarray} \label{define}
          v_i(t)^{\langle r \rangle } & = & \sum_{\ell=0}^d a_{i,\ell}^{(r)} \, v_{\ell } (t).
        \end{eqnarray}
     Hence

        \begin{eqnarray} \label{twovariable} 
           \sum_{i=0}^\infty v_i(t)^{\langle r \rangle } x^i & = &
                  \sum_{i=0}^\infty \sum_{\ell=0}^d a_{i,\ell}^{(r)} \frac{x^it^\ell}{(1-t)^{\ell}} \\
                                                             & = & 
        \label{twovariableI} \sum_{\ell=0}^d \Big( \sum_{i=0}^\infty a_{i,\ell}^{(r)} x^i \Big) \frac{t^\ell}{(1-t)^{\ell}}
        \end{eqnarray} 

        Next we derive a second expansion of the left hand side of  (\ref{twovariable}).
        First we derive an expansion of $v_i(t)^{\langle r \rangle }$ as a
         formal power series.  Clearly, for $i \geq 1$,
        $$ 
          v_i(t)
	  =  \frac{t^i}{(1-t)^{i}}
                                              =  \sum_{j=0}^\infty {j -1\choose i-1} t^j
        $$

        Therefore, for $i \geq 1$,
        $$v_i(t)^{\langle r \rangle } = \sum_{j=0}^\infty {r \, j -1 \choose i-1 } t^j.$$
        For $i = 0$ we have $v_0^{\langle r \rangle }(t) = 1 = v_0(t)$.

        The preceding expansion leads to the following identity. 
        \begin{eqnarray} 
           \nonumber
           \sum_{i=0}^\infty v_i(t)^{\langle r \rangle } x^i & = &
                     1 + \sum_{i=1}^\infty \sum_{j=0}^\infty {r \, j - 1\choose i-1} t^jx^i \\
           \nonumber
                                                             & = & 
                     1 + \sum_{j=0}^\infty \Big( x \sum_{i=1}^\infty  {r \, j -1 \choose i-1} x^{i-1}\Big) t^j \\
           \nonumber
                                                             & = & 
                     1 + \sum_{j=1}^\infty x (1+x)^{rj-1} t^j \\
           \label{twovariableII}
                                                             & = & 
                     1 +t(1+x)^{r-1} \frac{x}{1-t(1+x)^r}
        \end{eqnarray}

     Writing the formulas from (\ref{twovariableI}) and (\ref{twovariableII}) in terms of the variable $u = \frac{t}{1-t}$ or equivalently $t = \frac{u}{1+u}$ and comparing we obtain

        \begin{eqnarray*} 
            1 +u(1+x)^{r-1} \frac{x}{1+u-u(1+x)^r}  & = & \sum_{\ell=0}^d \Big( \sum_{i=0}^\infty a_{i,\ell}^{(r)} x^i \Big) u^\ell
        \end{eqnarray*}   

        Thus by

        \begin{eqnarray*}
           \frac{xu(1+x)^{r-1}}{1+u-u(1+x)^r} & = & \sum_{\ell=0}^\infty x(1+x)^{r-1}   ((1+x)^r - 1)^\ell u^{\ell +1}
        \end{eqnarray*}

        we obtain for $\ell \geq 1$

        \begin{eqnarray*}
            \sum_{i=0}^\infty a_{i,\ell}^{(r)} x^i & = & x(1+x)^{r-1} ((1+x)^r - 1)^{\ell -1}
        \end{eqnarray*}

        and $\sum_{i=0}^{\infty} a_{i,0}^{(r)} \, x^{i}=1$. Hence
        for any $\ell \geq 0$ and $i \geq 0$

        \begin{eqnarray} \label{formula}
            a_{i,\ell}^{(r)} & = & \sum_{{j_1+ \cdots + j_\ell = i} \atop {j_1, \ldots, j_\ell \geq 1}} {r-1 \choose j_1-1}{r \choose j_2} \cdots {r \choose j_\ell}
        \end{eqnarray}
    
        From this we conclude the following equalities which imply the assertion
        \begin{eqnarray*}
          \sum_{i=0}^d f_{i-1}^{\langle r \rangle} \, v_{i}(t) & = & \Hilb(k[\Delta^{\langle r\rangle }],t)  =  \Hilb(k[\Delta]^{\langle r\rangle },t) \\
                                                   & = & \Hilb(k[\Delta],t)^{\langle r \rangle} =  \sum_{i=0}^d f_{i-1} v_i(t)^{\langle r \rangle} \\
                                                   & = & \sum_{i=0}^d f_{i-1} \sum_{\ell=0}^d a_{i,\ell}^{(r)} v_{\ell }(t) \\
                                                   & = & \sum_{\ell=0}^d (\sum_{i=0}^d f_{i-1} a_{i,\ell}^{(r)}) \, v_{\ell } (t) \\
                                                   & = & \sum_{\ell=0}^d \Big(\sum_{i=\ell}^d f_{i-1} \sum_{{j_1+ \cdots + j_\ell = i} \atop {j_1, \ldots, j_\ell \geq 1}} {r-1 \choose j_{1}-1} {r \choose j_2} \cdots {r \choose j_\ell} \Big) \cdot
        \end{eqnarray*}

        The last equality follows from (\ref{formula}) and the fact that $a_{i,\ell}^{(r)} = 0$ for $i < \ell$.
     \end{proof}
 
     As an immediate consequence of Theorem \ref{limittheorem} we also get the following result.
         
     \begin{proposition} 
       For any  $d \geq 2$ there are real numbers $\alpha_1', \ldots, \alpha_{d-2}'$ such that
       for any $(d-1)$-dimensional simplicial complex $\Delta$
       there are $R > 0$ and sequences of complex numbers $(\beta_r^{(i)})_{r \geq 1}$, $1 \leq i \leq d$,
       such that :
       \begin{itemize}
          \item[(i)] $\beta_r^{(i)}$ is real for $r > R$ and $1 \leq i \leq d$ and strictly negative
                     for $r > R$ and $1 \leq i \leq d-1$.
          \item[(ii)] $\beta_r^{(i)} \rightarrow  \alpha ' _i$ for $r \rightarrow \infty$ and $1 \leq i \leq d-2$.
          \item[(iii)] $\beta_r^{(d-1)} \rightarrow -\infty$ for $r \rightarrow \infty$,
          \item[(iv)] $\beta_r^{(d)} \rightarrow -1$ for $r \rightarrow \infty$,
          \item[(v)] $f_{-1}^{\langle r \rangle} + \cdots + f_{d-1}^{\langle r \rangle}t^d =
                       \prod_{i=1}^d (1-\beta_r^{(i)}t)$, for $r>R$.
       \end{itemize}
       In particular, for $r>R$ the $f$-vector of $\Delta^{\langle r \rangle}$ is log-concave and unimodal.
     \end{proposition}

     In \cite[Theorem 3.1]{BrentiWelker} it is shown that for a simplicial complex with non-negative $h$-vector the $h$-polynomial of
     its first barycentric subdivision is real rooted. The example preceding the proof of Corollary \ref{algebralimitII}
     shows that such a result is not true for edgewise subdivision. More precisely, it shows that the $h$-polynomial of
     the second edgewise subdivision of the boundary of the $4$-simplex is of degree $4$ but has only two real roots. 

     \section*{Acknowledgment}
       We thank Aldo Conca and Bernd Sturmfels for suggesting the Veronese
       construction as an algebraic analog of barycentric subdivision of simplicial complexes. 
       We also thank Aldo Conca and Tim R\"omer for providing references to edgewise subdivision and its
       relation to Veronese rings.

\end{document}